\documentclass[12pt,a4paper]{amsart}
\usepackage{amssymb}
\usepackage{amsthm}
\usepackage{amsmath}
\usepackage{amsxtra}
\usepackage{latexsym}
\usepackage{mathrsfs}
\usepackage{amscd}
\usepackage[all,cmtip]{xy}
\usepackage[all]{xy}
\usepackage{enumitem}
\usepackage{xcolor}
\usepackage{comment}

\newtheorem{thm}{Theorem}[section]
\newtheorem{lem}[thm]{Lemma}
\newtheorem{conj}[thm]{Conjecture}

\theoremstyle{definition}
\newtheorem{dfn}[thm]{Definition}
\newtheorem{rmk}[thm]{Remark}

\numberwithin{equation}{section}
\newcommand{\var}{\overline}

\def\C{{\mathbb C}}
\def\Q{{\mathbb Q}}
\def\R{{\mathbb R}}
\def\Z{{\mathbb Z}}

\def\h{{\mathbf h}}
\def\D{{\mathbf D}}
\def\var{\overline}

\DeclareMathOperator{\NS}{NS}

\DeclareMathOperator{\Div}{Div}

\DeclareMathOperator{\Pic}{Pic}

\usepackage{mathtools}
\mathtoolsset{showonlyrefs=true}

\numberwithin{equation}{section}

\newenvironment{parts}[0]{%
  \begin{list}{}%
    {\setlength{\itemindent}{0pt}
     \setlength{\labelwidth}{1.5\parindent}
     \setlength{\labelsep}{.5\parindent}
     \setlength{\leftmargin}{2\parindent}
     \setlength{\itemsep}{0pt}
     }%
   }%
  {\end{list}}
\newcommand{\Part}[1]{\item[\upshape#1]}

\title[Growth rate and DML]
{Growth rate of ample heights and the Dynamical Mordell-Lang conjecture}
\author{Kaoru Sano}
\address{Department of Mathematics, Faculty of Science, Kyoto University, Kyoto 606-8502, Japan}
\email{ksano@math.kyoto-u.ac.jp}
\begin{document}
\maketitle
\begin{abstract}
We provide an explicit formula on the growth rate of ample heights of rational points under iteration of endomorphisms of smooth projective varieties over number fields.
As an application, we give a positive answer to a variant of the Dynamical Mordell-Lang conjecture
for pairs of \'etale endomorphisms, which is also a variant of the original one stated by Bell, Ghioca, and Tucker in their monograph.
\end{abstract}

\tableofcontents

\section{Introduction}\label{intro}
%%=============記号のセッティング===============
Let $X$ be a smooth projective variety over $\var{\Q}$
and $f\colon X\longrightarrow X$ a surjective endomorphism of $X$ over $\var{\Q}$.
Here, an endomorphism simply means a self-morphism.
%%==============算術次数の定義=============
Fix an ample divisor $H$ on $X$ over $\var{\Q}$
and take a Weil height function $h_H\colon X(\var{\Q})\longrightarrow \R$
associated with $H$.
For a point $P\in X(\var{\Q})$,
the {\it arithmetic degree} $\alpha_f(P)$ of $f$ at $P$ is defined by
\[
\alpha_f(P) :=\lim_{n\to\infty} \max\{1, h_H(f^n(P))\}^{1/n}.
\]
It is known that the arithmetic degree $\alpha_f(P)$ is well-defined, and independent of the choice of $H$ and $h_H$; see Remark \ref{rmk: arith}.
%%===================主定理===============
By using the arithmetic degree, we can describe the growth rate of the ample heights
of rational points under the iteration of the endomorphism as follows.
\begin{thm}\label{Theorem: order of height growth}
Let $X$ be a smooth projective variety over $\var{\Q}$
and $f \colon X \longrightarrow X$ a surjective endomorphism of $X$
over $\var{\Q}$.
Let $H$ be an ample divisor on $X$ over $\var{\Q}$.
Then for any point $P\in X(\var{\Q})$ with $\alpha_f(P)>1$,
there is a non-negative integer $t_f(P)\in \Z_{\geq 0}$,
positive real numbers $C_0, C_1>0$, and an integer $N_0$
such that
\[
C_0 n^{t_f(P)}\alpha_f(P)^n < h_H(f^n(P)) < C_1 n^{t_f(P)}\alpha_f(P)^n
\]
for all $n\geq N_0$.
\end{thm}

As an application of Theorem \ref{Theorem: order of height growth},
we prove a variant of the Dynamical Mordell-Lang conjecture (see Section \ref{section: Backgrounds} for details).

%%====================DML type theorem==================
\begin{thm}\label{Theorem: DML type theorem}
Let $X$ be a smooth projective variety over $\var{\Q}$, and
$f,g\colon X\longrightarrow X$ \'etale endomorphisms of $X$ over $\var{\Q}$.
Let $P,Q\in X(\var{\Q})$ be points satisfying the following two conditions:
\begin{itemize}
\item $\alpha_f(P)^p=\alpha_g(Q)^q>1$ for some $p,q \in \Z_{\geq 1}$, and
\item $t_f(P)=t_g(Q)$, where $t_f(P)$ and $t_g(Q)$ are as in Theorem \ref{Theorem: order of height growth}.
\end{itemize}
Then the set
\[
S_{f,g}(P,Q):= \{ (m,n)\in \Z_{\geq 0}\times \Z_{\geq 0}\ |\ f^m(P)=g^n(Q)\}
\]
is a finite union of the sets of the form
\[
\{(a_i+b_i\ell,c_i+d_i\ell)\ |\ \ell\in\Z_{\geq 0}\}
\]
for some non-negative integers $a_i,b_i,c_i,d_i\in \Z_{\geq 0}$.
\end{thm}

We now briefly sketch the plan of this paper.
In Section \ref{section: Notation}, we fix some notation.
In Section \ref{section: proof of order},
we prove Theorem \ref{Theorem: order of height growth}.
In Section \ref{section: Backgrounds}, we provide backgrounds
of Theorem \ref{Theorem: DML type theorem},
and recall known results related to this theorem.
In Section \ref{section: proof of DML type},
we prove Theorem \ref{Theorem: DML type theorem}.
It seems plausible that we can generalize the assertion of
Theorem \ref{Theorem: DML type theorem} further.
We give a conjecture
(Conjecture \ref{Conjecture: DML type conjecture})
generalizing Theorem \ref{Theorem: DML type theorem},
and some evidence in Section \ref{section: examples}.
Furthermore, to see that the results given
in Section \ref{section: examples} support our conjecture,
we give a definition of the double canonical height in a special case
(see the proof of Theorem \ref{Theorem: DML type finiteness}).

\section{Notation and definitions}\label{section: Notation}
Let $X$ be a smooth projective variety, and $f\colon X\longrightarrow X$ a surjective endomorphism of $X$
both over $\var{\Q}$.
We denote the group of divisors, the Picard group, and the N\'eron-Severi group
by $\Div(X)$, $\Pic(X)$, and $\NS(X)$, respectively.
For a divisor $D$ on $X$, fix a Weil height function
$h_D\colon X(\var{\Q})\longrightarrow \R$ associated with $D$;
see \cite[Theorem 8.3.2]{HS}.
For a $\C$-divisor $D=\sum_{i=1}^r a_i D_i\ (a_i\in\C, D_i\in \Div(X))$,
we put $h_D:=\sum_{i=1}^r a_i h_{D_i}$.

\begin{dfn}
Let $H$ be an ample divisor on $X$ over $\var{\Q}$,
and $P\in X(\var{\Q})$ a point.
The {\it arithmetic degree} $\alpha_f(P)$ of $f$ at $P$ is defined by
\[
\alpha_f(P):=\lim_{n\to \infty} \max\{ 1, h_H(f^n(P))\}^{1/n}.
\]
\end{dfn}

\begin{rmk}\label{rmk: arith}
The existence of the arithmetic degree for surjective endomorphisms
is proved by Kawaguchi and Silverman
(see \cite[Theorem 3]{KS2}).
They also proved that $\alpha_f(P)$ is independent of the choice of $H$ and $h_H$.
\end{rmk}
\begin{dfn}
For a column vector $v={}^t(x_0,\ldots ,x_N)\in \C^{N+1}$, we set 
\[\|v\|:=\max_{0\leq i\leq N}\{ |x_i| \}.\]
For a square matrix $A=(a_{i,j})\in M_{N+1}(\C)$,
we similarly set
\[\|A\|:=\max_{0\leq i,j\leq N}\{|a_{i,j}|\}.\]
We frequently use the following inequality
\[
\| Av \|\leq (N+1)\|A\|\cdot \|v\|.
\]
\end{dfn}
\begin{dfn}\label{dfn: asymp}
For sequences $(a_n)_{n\geq 0}$ and $(b_n)_{n\geq 0}$ of positive real numbers, we write
$a_n \preceq b_n$
if there is a positive real number $C>0$ and an integer $N_0$ such that the inequality
$a_n \leq C\cdot b_n$
holds  for all $n\geq N_0$.
If both $a_n\preceq b_n$ and $b_n\preceq a_n$ hold, we write $a_n \asymp b_n$.
\end{dfn}

\section{Proof of Theorem \ref{Theorem: order of height growth}}\label{section: proof of order}
First, we give some lemmata on linear algebra.
Then, we shall prove Theorem \ref{Theorem: order of height growth}.

%%=================ジョルダン標準形行列の定義==============
For a non-negative integer $\ell\in\Z_{\geq 0}$ and a complex number $\lambda\in\C$,
let
\[ \Lambda:=\left(
\begin{array}{ccccc}
\lambda & & & & \\
1 &\lambda & &O & \\
 &1 &\ddots & & \\
O & &\ddots &\lambda & \\
 & & &1 &\lambda
\end{array}
\right)
\]
be the Jordan block matrix of the size $(\ell+1)\times (\ell+1)$.
We put $N:=\Lambda-\lambda I$.

\begin{lem}\label{Lemma: asymp of Jordan matrix}
\begin{parts}
\Part{(a)}
When $|\lambda|\geq 1$, we have
$\|\Lambda^n\| \asymp n^\ell |\lambda|^n.$
\Part{(b)} When $|\lambda|<1$, we have $\| \Lambda^n\| \leq n^\ell |\lambda|^{n-\ell}$.
\end{parts}
\end{lem}
\begin{proof}
Note that $\binom{n}{k}\asymp n^k$ and $\binom{n}{k}\leq n^k$.
Then both assertions follow from the following equalities:
\begin{align*}
\|\Lambda^n \| &= \|(\lambda I+N)^n\|\\
&=\left\| \sum_{k=0}^n \binom{n}{k}\lambda^{n-k}N^k\right\|\\
&=\left\| \sum_{k=0}^\ell \binom{n}{k}\lambda^{n-k}N^k\right\| &\text{since } N^{\ell+1}=0\\
&=\max_{0\leq k \leq \ell} \binom{n}{k}|\lambda|^{n-k}.
\end{align*}
\end{proof}

\begin{lem}\label{Lemma: big heights}
Assume $|\lambda|>1$.
For a non-zero column vector $v={}^t(x_0,\ldots ,x_\ell)\in \C^{\ell+1}\setminus \{0\}$,
we have
\[
\|\Lambda^n v\| \asymp n^t |\lambda |^n,
\]
where we put
\[
t:=\ell-\min\{ i\ |\ 0\leq i \leq \ell, \ x_i\neq 0\}.
\]
\end{lem}

\begin{proof}
We may assume $t=\ell$, so $x_0\neq 0$.
For a negative integer $j<0$, we set $x_j := 0$.
The asymptotic inequality
$\| \Lambda^n v\|\preceq n^\ell |\lambda|^n$
follows from the following inequalities.
\begin{align*}
|(\Lambda^n v)_j|
&= \left| \sum_{k=0}^\ell \binom{n}{k}\lambda^{n-k}(N^k v)_j \right|\\
&= \left| \sum_{k=0}^\ell \binom{n}{k}\lambda^{n-k}x_{j-k} \right|\\
%&= \left| \sum_{k=0}^j \binom{n}{k}\lambda^{n-k}x_{j-k} \right|\\
%&\leq \sum_{k=0}^j \left| \binom{n}{k}\lambda^{n-k}x_{j-k}\right|\\
&\leq \sum_{k=0}^j n^k\cdot|\lambda|^n \cdot |x_{j-k}|\\
&\leq (\ell+1)\cdot n^\ell \cdot |\lambda|^n \cdot \| v \|
\end{align*}
The converse asymptotic inequality
$\| \Lambda^n v\|\succeq n^\ell |\lambda|^n$
follows from the following asymptotic inequalities.
\begin{align*}
|(\Lambda^n v)_\ell |
&= \left|\sum_{k=0}^\ell\binom{n}{k}\lambda^{n-k}x_{\ell-k}\right|\\
&\geq \binom{n}{\ell}|\lambda^{n-\ell}x_0|
- \sum_{k=0}^{\ell-1}\left| \binom{n}{k}\lambda^{n-k}x_{\ell-k}\right|\\
&\succeq n^\ell |\lambda|^n- n^{\ell-1}|\lambda|^n\\
&\succeq n^\ell |\lambda|^n.
\end{align*}
Hence we conclude $\| \Lambda^n v\|\asymp n^\ell |\lambda|^n$.
\end{proof}

%============================proof of Theorem 1.1==========================
Let notation be the same as in Theorem \ref{Theorem: order of height growth}.
Let $V_H$ be the $\Q$-vector subspace of $\Pic(X)_\Q := \Pic(X)\otimes _\Z \Q$
spanned by the set $\{ (f^n)^\ast H\ |\ n\geq 0\}$,
and $\var{V_H}$ the image of $V_H$ in $\NS(X)_\Q:= \NS(X)\otimes_\Z \Q$.
It is known that $V_H$ is a finite dimensional $\Q$-vector space
(see the proof of \cite[Theorem 3]{KS2}).
We decompose the $\C$-vector space $(V_H)_\C:=V_H\otimes_\Z \C$
into Jordan blocks with respect to the $\C$-linear map
$f^\ast \colon (V_H)_\C \longrightarrow (V_H)_\C$:
\[
(V_H)_\C =\bigoplus_{i=1}^\tau V_i.\]
For each $1\leq i\leq \tau$, let $\lambda_i\in\C$ be the eigenvalue of $f^\ast |_{V_i}$.
By changing the order of the Jordan blocks if necessary, we may assume
\[
|\lambda_1|\geq |\lambda_2| \geq \cdots \geq |\lambda_\sigma| > 1 \geq |\lambda_{\sigma+1}|\geq\cdots \geq |\lambda_{\tau}|
\]
for some $0\leq \sigma \leq \tau$.
We put $\ell_i :=\dim_\C V_i-1$.

We take a $\C$-basis $\{ D_{i,j}\}_{0\leq j\leq \ell_i}$ of $V_i$
satisfying the following linear equivalences:
\begin{equation}\label{eqn: linearly equivalence}
f^\ast D_{i,j}\sim D_{i,j-1}+\lambda_i D_{i,j} \quad (0\leq j \leq \ell_i),
\end{equation}
where we set $D_{i,-1}:=0$.
For each $1\leq i \leq \sigma$, let
\[\widehat{h}_{D_{i,j}}\colon X(\var{\Q}) \longrightarrow \C\quad (0\leq j \leq \ell_i)
\]
be unique functions satisfying the normalization condition
\[
\widehat{h}_{D_{i,j}}=h_{D_{i,j}}+O(1)\quad (0\leq j\leq \ell_i)
\]
and the functional equation
\[
\widehat{h}_{D_{i,j}}\circ f = \widehat{h}_{D_{i,j-1}}+\lambda_i \widehat{h}_{D_{i,j}}
\quad (0\leq j\leq \ell_i)
\]
(see \cite[Theorem 5]{KS2} for the existence of such functions).

For each $1\leq i \leq \tau$, we set
\[
\h_{\D_i}:={}^t(h_{D_{i,0}},h_{D_{i,1}},\ldots , h_{D_{i,\ell_i}}).\]
For each $1\leq i \leq \sigma$, we set
\[
\widehat{\h}_{\D_i}:={}^t(\widehat{h}_{D_{i,0}},\widehat{h}_{D_{i,1}},\ldots , \widehat{h}_{D_{i,\ell_i}}).
\]
Let $\Lambda_i$ be the Jordan block matrix of size $(\ell_i+1)\times (\ell_i+1)$
associated with the eigenvalue $\lambda_i$.

\begin{lem}\label{Lemma: small heights}
For each $\sigma +1\leq i\leq \tau$ and a point $P\in X(\var{\Q})$,
we have
\[
\|\h_{\D_i}(f^n(P))\|\preceq n^{\ell_i+1}.
\]
\end{lem}
\begin{proof}
By \eqref{eqn: linearly equivalence}, there is a positive real number $C_0>0$
such that the inequality
\begin{equation}
\| \h_{\D_i}\circ f-\Lambda_i \cdot \h_{\D_i} \| \leq C_0
\end{equation}
holds on $X(\var{\Q})$.
Therefore, there is a positive real number $C_1>0$ such that for every point
$P\in X(\var{\Q})$, the following inequalities hold:
\begin{align}
&\quad\| \h_{\D_i}\circ f^n (P)\|-\|\Lambda_i^n \cdot \h_{\D_i}(P)\|\\
&\leq \|\h_{\D_i}\circ f^n(P) -\Lambda_i^n\cdot \h_{\D_i}(P)\|\\
&\leq \sum_{k=0}^{n-1}\|\Lambda_i^k\cdot (\h_{\D_i}\circ f^{n-k}(P))
-\Lambda_i^{k+1}\cdot(\h_{\D_i}\circ f^{n-k-1}(P)) \|\\
&\leq \sum_{k=0}^{n-1} (\ell_i+1)\| \Lambda_i^k \|\cdot \| \h_{\D_i}(f^{n-k}(P))-
\Lambda_i \cdot \h_{\D_i}( f^{n-k-1}(P))\|\\
&\leq \sum_{k=0}^{n-1}(\ell_i+1)\| \Lambda_i^k\|C_0\\
&\leq \sum_{k=0}^{n-1}(\ell_i+1) \cdot n^{\ell_i} \cdot |\lambda_i|^{k-\ell_i} \cdot C_0
&\hspace{-2em}\text{by Lemma \ref{Lemma: asymp of Jordan matrix}}\\
&\leq C_1 n^{\ell_i+1}
&\text{because } |\lambda_i|\leq 1.
\end{align}
Furthermore, we have
\begin{align}
\| \Lambda_i^n\cdot \h_{\D_i}(P) \|
&\leq (\ell_i+1)\cdot\|\Lambda_i^n\|\cdot \| \h_{\D_i}(P)\|\\
&\leq (\ell_i+1)\cdot n^{\ell_i}\cdot |\lambda_i|^{n-\ell_i}\cdot\| \h_{\D_i}(P)\|
&\text{by Lemma \ref{Lemma: asymp of Jordan matrix}}\\
&\leq (\ell_i+1)\cdot n^{\ell_i}\cdot\| \h_{\D_i}(P)\|
&\text{because } |\lambda_i|\leq 1.
\end{align}
Combining these inequalities, the assertion is proved.
\end{proof}

\begin{lem}[{see \cite[Lemma 18]{KS2}}]\label{Lemma: ample is strong}
For a $\C$-divisor $D$ on $X$
and a point $P\in X(\var{\Q})$,
we have
\[
\max\{1,h_H(f^n(P))\} \succeq |h_D(f^n(P))|.
\]
\end{lem}
\begin{proof}
The assertion is obviously true when the forward $f$-orbit of $P$ is a finite set.
Hence, we may assume the forward $f$-orbit of $P$ is an infinite set.
Thus we may assume
\begin{equation}\label{eqn: infty}
h_H(f^n(P))\to \infty \quad (\text{as }n\to \infty).
\end{equation}
Write $D= D_r+ \sqrt{-1}D_c$, where $D_r$ and $D_c$ are $\R$-divisors on $X$.
By the triangle inequality, it is enough to prove the assertion for $D_r$ and $D_c$.
Thus we may assume $D$ is an $\R$-divisor.

Take a sufficiently large positive real number $C>0$
such that $CH\pm D$ are ample.
The function $C\cdot h_H- |h_D|$ is bounded below on $X(\var{\Q})$.
There is a (not necessarily positive) real number $C'\in \R$ satisfying
\begin{align}\label{eqn: ample is strong}
Ch_H(f^n(P))- | h_D(f^n(P)) | \geq C'
\end{align}
for all $P\in X(\var{\Q})$.
By \eqref{eqn: infty} and \eqref{eqn: ample is strong}, the assertion follows.
\end{proof}
\begin{proof}[Proof of Theorem \ref{Theorem: order of height growth}]
Write $H$ in terms of the $\C$-linear bases of $(V_H)_\C$:
\[
H=\sum_{i=1}^{\tau}\sum_{j=0}^{\ell_i} c_{i,j} D_{i,j} \quad (c_{i,j}\in \C).
\]
Let $P\in X(\var{\Q})$ be a point with $\alpha_f(P)>1$.
If $\widehat{\h}_{\D_i}(P)=0$ for all $1\leq i \leq \sigma$, we have the following inequalities:
\begin{align}
&\hphantom{\leq}h_H(f^n(P))\\
&\leq \left|\sum_{i=1}^{\sigma}\sum_{j=0}^{\ell_i} c_{i,j}\widehat{h}_{D_{i,j}}(f^n(P))\right| +
\left| \sum_{i=\sigma+1}^{\tau}\sum_{j=0}^{\ell_i} c_{i,j}h_{D_{i,j}}(f^n(P))\right| +O(1)\\
&= \left| \sum_{i=\sigma+1}^{\tau}\sum_{j=0}^{\ell_i}c_{i,j}h_{D_{i,j}}(f^n(P))\right| +O(1)\\
&\preceq n^{\max_i \ell_i+1} &\hspace{-2em}\text{by Lemma \ref{Lemma: small heights}.}
\end{align}
Thus we get
\[
\alpha_f(P)\leq\lim_{n\to\infty}n^{(\max_i \ell_i+1)/n}=1.
\]
But this contradicts $\alpha_f(P)>1$.
Hence we get $\widehat{\h}_{\D_i}(P)\neq 0$ for some $1\leq i \leq \sigma$.

We set
\[
\lambda := \max\{ |\lambda_i|\ |\ 1\leq i\leq \sigma,\ \widehat{\h}_{\D_i}(P)\neq 0\}.
\]
If $\widehat{h}_{D_{i,j}}(P)= 0$, we set $t_{f,i}(P):=-\infty$.
Otherwise, we set
\[
t_{f,i}(P) := \ell_i-\min\{ j\ |\ 0\leq j \leq \ell_i,\ \widehat{h}_{D_{i,j}}(P)\neq 0\}.
\]
Finally, we set
\[
t_f(P) :=\max\{ t_{f,i}(P)\ |\ \lambda=|\lambda_i|\}.
\]
Since $\widehat{\h}_{\D_i}(P)\neq 0$ holds for some $1\leq i \leq \sigma$,
we get $t_f(P)\neq -\infty$.
It is enough to prove
\[
h_H(f^n(P))\asymp n^{t_f(P)}\lambda^n.
\]
Note that $\lambda=\alpha_f(P)$ follows from this asymptotic inequality.
We have
\begin{align}\label{eqn: both bound}
\hspace{3em}h_H(f^n(P))
&\leq \left| \sum_{i=1}^{\tau}\sum_{j=0}^{\ell_i}c_{i,j}h_{D_{i,j}}(f^n(P))\right| +O(1)\\
&\leq \sum_{i=1}^{\tau}\sum_{j=0}^{\ell_i}|c_{i,j}|\cdot |h_{D_{i,j}}(f^n(P))| +O(1)\\
&\preceq \sum_{i=1}^{\tau}\sum_{j=0}^{\ell_i}|c_{i,j}|\cdot h_H(f^n(P))
&\text{by Lemma \ref{Lemma: ample is strong}}\\
&\preceq h_H(f^n(P)).
\end{align}
By the equality $\widehat{\h}_{\D_i}(f(P))=\Lambda_i\widehat{\h}_{\D_i}(P)$ and Lemma \ref{Lemma: big heights}, we get
\begin{equation}\label{eqn: big heights}
\|\widehat{\h}_{\D_i}(f^n(P))\| \asymp n^{t_{f,i}(P)}|\lambda_i|^n\quad (1\leq i\leq \sigma).
\end{equation}
Combining Lemma \ref{Lemma: small heights} and \eqref{eqn: big heights},
we conclude
\[
\sum_{i=1}^{\tau}\sum_{j=0}^{\ell_i}|c_{i,j}|\cdot |h_{D_{i,j}}(f^n(P))| \asymp n^{t_f(P)}\lambda^n.
\]
The assertion follows from this asymptotic equality and \eqref{eqn: both bound}.
\end{proof}

%=============DML type theorem の章========================
\section{Backgrounds and general conjectures} \label{section: Backgrounds}
Theorem \ref{Theorem: DML type theorem} gives a positive answer
to a variant of the Dynamical Mordell-Lang conjecture for pairs of \'etale endomorphisms, which is
a variant of the original one stated by Bell, Ghioca, and Tucker
(see \cite[Question 5.11.0.4]{BGT2}).
In \cite{GTZ1} and \cite{GTZ2}, Ghioca, Tucker, and Zieve studied similar problems for polynomial maps and got deeper results.
Moreover, they introduced some of reductions including Lemma \ref{Lemma: sufficient condition} we use.
In \cite{GN}, Ghioca and Nguyen also studied similar problems for self-maps of semi-abelian varieties.

First, we recall a version of the Dynamical Mordell-Lang conjecture.
Note that there are several variants of the Dynamical Mordell-Lang conjecture.
Many results are obtained in various situations (see \cite{BGT2} for details).
\begin{conj}[{the Dynamical Mordell-Lang conjecture \cite[Conjecture 1.7]{GT}}]\label{Conjecture: DML}
%%%%%%%%%%%%%%%%%%%%%%%%%%%%%%%%%%%%%%%%%%%%%%%%%%%%%%%%%%%参考文献を書く
Let $X$ be a quasi-projective variety over $\C$.
Let $f\colon X\longrightarrow X$ be an endomorphism of $X$ over $\C$.
For any $\C$-rational point $P\in X(\C)$ and any closed subvariety $Y\subset X$,
the set
\[
S_f(P,Y):= \{ n\in \Z_{\geq 0} \ |\ f^n(P)\in Y(\C)\}
\]
is a union of finitely many sets of the form
\[
\{ a_i+b_im\ |\ m\in\Z_{\geq 0}\}
\]
for some non-negative integers $a_i,b_i \in \Z_{\geq 0}$.
\end{conj}
In Section \ref{section: proof of DML type},
we shall use the following result proved by Bell, Ghioca, and Tucker,
which is a special case of the Dynamical Mordell-Lang conjecture.
\begin{thm}[{\cite[Theorem 1.3]{BGT1}}]\label{Theorem: DML for etale}
If $f\colon X\longrightarrow X$ is an \'etale endomorphism,
Conjecture \ref{Conjecture: DML} holds.
\end{thm}

Furthermore, in \cite[Question 5.11.0.4]{BGT2}, the following conjecture is stated.
\begin{conj}[{\cite[Question 5.11.0.4]{BGT2}}]\label{Conjecture: DML type by BGT}
Let $X$ be a projective variety over $\var{\Q}$.
Let $H$ be an ample $\R$-divisor on $X$ over $\var{\Q}$.
Let $f,g\colon X\longrightarrow X$ be \'etale endomorphisms of $X$
over $\var{\Q}$ such that $f^\ast H \equiv \delta_f H$
and $g^\ast H \equiv \delta_g H$ hold in $\NS(X)_\R$ for some $\delta_f,\delta_g\in \R_{> 1}$.
Then for any points $P,Q\in X(\var{\Q})$, the set
\[
S_{f,g}(P,Q):= \{ (m,n)\ |\ f^m(P)=g^n(Q)\}
\]
is a union of finitely many sets of the form
\[
\{ (a_i+b_i\ell, c_i+d_i\ell)\ |\ \ell\in\Z_{\geq 0}\}
\]
for some non-negative integers $a_i,b_i,c_i,d_i\in\Z_{\geq 0}$
\end{conj}
Bell, Ghioca, and Tucker proved a special case of Conjecture \ref{Conjecture: DML type by BGT}.
\begin{thm}[{\cite[Theorem 5.11.0.1]{BGT2}}]\label{Theorem: DML type theorem by BGT}
Conjecture \ref{Conjecture: DML type by BGT} holds if $\delta_f=\delta_g$.
\end{thm}
\begin{rmk}[{see \cite[Theorem 2 (a), Proposition 7]{KS1} for details}]\label{Remark: polarized case}
If an ample $\R$-divisor $H$ satisfies $f^\ast H \equiv dH$
in $\NS(X)_\R$ for some $d\in\R_{>1}$,
the limit
\[
\widehat{h}_{f,H}(P):=\lim_{n\to\infty}\frac{h_H(f^n(P))}{d^n}
\]
converges for all $P\in X(\var{\Q})$ and satisfies
\begin{equation}\label{eqn: polarized}
\widehat{h}_{f,H}(f(P)) = d\widehat{h}_{f,H}(P)
\end{equation}
and
\begin{equation}\label{polarized order}
\widehat{h}_{f,H}- h_H = O(\sqrt{h_H}).
\end{equation}
The function $\widehat{h}_{f,H}$ is called the canonical height.
Furthermore, the following conditions are equivalent to each other:
\begin{itemize}
\item $\alpha_f(P) >1$,
\item $\alpha_f(P)=d$,
\item $\widehat{h}_{f,H}(P) \neq 0$, and
\item the forward $f$-orbit of $P$ is an infinite set.
\end{itemize}
\end{rmk}
In the setting of Theorem \ref{Theorem: DML type theorem by BGT},
when the forward $f$-orbit of $P$ or the forward $g$-orbit of $Q$ is finite,
the assertion of Theorem \ref{Theorem: DML type theorem by BGT} is obviously true.
By Remark \ref{Remark: polarized case},
we have
\[
\alpha_f(P)=\delta_f=\delta_g=\alpha_g(Q)>1\]
in the remaining case.
Thus, when $X$ is smooth,
Theorem \ref{Theorem: DML type theorem} is a generalization of
Theorem \ref{Theorem: DML type theorem by BGT}.
The assumption of Conjecture \ref{Conjecture: DML type by BGT} seems too strong.
We propose a more general conjecture as follows.

\begin{conj}\label{Conjecture: DML type conjecture}
Let $X$ be a smooth projective variety over $\var{\Q}$, and
let $f,g\colon X\longrightarrow X$ be \'etale endomorphisms of $X$ over $\var{\Q}$.
For points $P,Q\in X(\var{\Q})$ with $\alpha_f(P)>1$ and $\alpha_g(Q)>1$,
the following statements hold.
\begin{parts}
\Part{(a)} The set $S_{f,g}(P,Q)$
is a finite union of the sets of the form
\[
\{(a_i+b_i\ell,c_i+d_i\ell)\ |\ \ell\in\Z_{\geq 0}\}
\]
for some non-negative integers $a_i,b_i,c_i, d_i \in \Z_{\geq 0}$.
\Part{(b)}
If $\log_{\alpha_f(P)}\alpha_g(Q)$ is irrational or $t_f(P)\neq t_g(Q)$,
the set $S_{f,g}(P,Q)$ is finite.
\end{parts}
\end{conj}
The part $(b)$ asserts that the hypotheses from Theorem \ref{Theorem: DML type theorem} regarding $\delta_f(P)$,  $\delta_g(Q)$, $t_f(P)$, $t_g(Q)$ must met if the set $S_{f,g}(P,Q)$ were infinite.
In Section \ref{section: examples}, we give some examples
of endomorphisms for which Conjecture \ref{Conjecture: DML type conjecture} (b) hold.

\section{Proof of Theorem \ref{Theorem: DML type theorem}}\label{section: proof of DML type}
In this section we prove Theorem \ref{Theorem: DML type theorem}.

\begin{lem}\label{Lemma: reduction}
To prove Theorem \ref{Theorem: DML type theorem}, we may assume $p=q=1$.
\end{lem}

\begin{proof}
We see that
\begin{align}
S_{f,g}(P,Q)= \bigcup_{\substack{0\leq i \leq p-1\\ 0\leq j \leq q-1}}
\{ (i+pm,j+qn)\ |\ (m,n)\in S_{f^p,g^q}(f^i(P),g^j(Q))\}.
\end{align}
Thus to prove Theorem \ref{Theorem: DML type theorem},
it is enough to prove it for $f^p$ and $g^q$.
By using Theorem \ref{Theorem: order of height growth} twice, we have
\begin{align}
n^{t_{f^p}(P)}\alpha_{f^p}(P)^n &\asymp h_H(f^{pn}(P))\\
&\asymp (pn)^{t_f(P)}\alpha_f(P)^{pn}\\
&\asymp n^{t_f(P)}\alpha_f(P)^{pn}.
\end{align}
Similarly, we obtain
\[n^{t_{g^q}}(Q) \alpha_{g^q}(Q)^n\asymp n^{t_g(Q)} \alpha_g(Q)^{qn}.\]
Hence combining with the assumption of
Theorem \ref{Theorem: DML type theorem} for $f$ and $g$,
we get
\begin{align}
t_{f^p}(P) &= t_f(P) = t_g(Q)= t_{g^q}(Q),\\
\alpha_{f^p}(P) &= \alpha_f(P)^p=\alpha_g(Q)^q= \alpha_{g^q}(Q).
\end{align}
Hence our assertion follows.
\end{proof}

\begin{lem}\label{Lemma: sufficient condition}
To prove Theorem \ref{Theorem: DML type theorem} in the case $p=q=1$,
it is enough to prove
\[
\sup_{(m,n)\in S_{f,g}(P,Q)}|m-n|<\infty .
\]
\end{lem}
\begin{proof}
We set
\[M:= \sup_{(m,n)\in S_{f,g}(P,Q)}|m-n|.\]
Then we have
\begin{align}
S_{f,g}(P,Q) &=
\bigcup_{0\leq k \leq M} \{ (m,m+k)\in S_{f,g}(P,Q)\ |\ m\in \Z_{\geq 0}\}\\
&\quad \cup
\bigcup_{1\leq k \leq M} \{(n+k,n)\in S_{f,g}(P,Q)\ |\ n\in \Z_{\geq 0}\}.
\end{align}
Let $f\times g\colon X\times X \longrightarrow X\times X$
be the product of the endomorphisms $f,g$.
Let $\Delta\subset X\times X$ be the diagonal.
Then we have
\begin{align}
(m,m+k)\in S_{f,g}(P,Q) &\Leftrightarrow f^m(P)=g^{m+k}(Q)\\
&\Leftrightarrow (f\times g)^m(P,g^k(Q))\in \Delta.
\end{align}
Since $f\times g$ is \'etale, the Dynamical Mordell-Lang conjecture is true for
$f\times g$ by Theorem \ref{Theorem: DML for etale}.
Hence the set
\[
\{ (m,m+k)\in S_{f,g}(P,Q)\ |\ m\in \Z_{\geq 0}\}
\]
is a finite union of the sets of the form
\[\{(a_i+b_i\ell, a_i+b_i\ell+k)\ |\ \ell\in\Z_{\geq 0} \}\]
for some non-negative integers $a_i,b_i \in \Z_{\geq 0}$.
Similarly, the set
\[
\{ (n+k,n)\in S_{f,g}(P,Q)\ |\ n\in \Z_{\geq 0}\}
\]
is a finite union of the sets of the form
\[
\{(c_i+d_i\ell+k,c_i+d_i\ell)\ |\ \ell\in\Z_{\geq 0}\}
\]
for some non-negative integers $c_i,d_i\in \Z_{\geq 0}$.
Thus the assertion follows.
\end{proof}

\begin{rmk}
Lemma \ref{Lemma: sufficient condition} is the only part where
the assumption of the \'etaleness of $f,g$ is used.
So Theorem \ref{Theorem: DML type theorem} is true
if the Dynamical Mordell-Lang conjecture (Conjecture \ref{Conjecture: DML}) is true for the endomorphism
\[
f^p\times g^q\colon X\times X\longrightarrow X\times X\]
and the diagonal $\Delta\subset X\times X$.
\end{rmk}

\begin{proof}[{Proof of Theorem \ref{Theorem: DML type theorem}}]
By Lemma \ref{Lemma: reduction}, we may assume $p=q=1$.
By Theorem \ref{Theorem: order of height growth},
there are positive real numbers $C_0,C_1,C_2,C_3>0$ such that the inequalities
\begin{align}
C_0 m^{t_f(P)}\alpha_f(P)^m &\leq h_H(f^m(P))\leq C_1m^{t_f(P)}\alpha_f(P)^m,\\
C_2 n^{t_g(Q)}\alpha_g(Q)^n &\leq h_H(g^n(Q))\leq C_3 n^{t_g(Q)}\alpha_g(Q)^n
\end{align}
hold except for finitely many $m$ and $n$.
Suppose $f^m(P)=g^n(Q)$ with $m\geq n$.
Then we get
\begin{equation}
(m/n)^{t_f(P)} \alpha_f(P)^{m-n}\leq C_3/C_0
\end{equation}
because $t_f(P)=t_g(Q)$ and $\alpha_f(P)=\alpha_g(Q)$ by assumption.
Hence we get
\begin{equation}\label{eqn: gap}
\alpha_f(P)^{m-n} \leq C_3/C_0.
\end{equation}
Since $\alpha_f(P)>1$, the inequality \eqref{eqn: gap} holds
only for finitely many values of $m-n$.
By the same argument for the case $m<n$, we conclude
\[\sup_{(m,n)\in S_{f,g}(P,Q)}|m-n|<\infty.\]
Hence by Lemma \ref{Lemma: sufficient condition},
the proof of Theorem \ref{Theorem: DML type theorem} is complete.
\end{proof}

%%============================予想の証拠の章========================
\section{Some evidence for Conjecture \ref{Conjecture: DML type conjecture}}\label{section: examples}
In this section,  we prove the following theorem which gives some evidence
for Conjecture \ref{Conjecture: DML type conjecture} (b).
\begin{thm}\label{Theorem: DML type finiteness}
Let $X$ be a smooth projective variety over $\var{\Q}$.
Let $f,g\colon X\longrightarrow X$ be surjective endomorphisms
on $X$ over $\var{\Q}$.
Assume that $f$ commutes with $g$, and
there is an ample $\R$-divisor $H$ on $X$ over $\var{\Q}$
such that $f^\ast H \equiv dH$ in $\NS(X)_\R$ for some $d\in\R_{>1}$.
Let $P,Q\in X(\var{\Q})$ be points with $\alpha_f(P)>1$ and $\alpha_g(Q)>1$.
Assume that $\log_{\alpha_f(P)}\alpha_g(Q)$ is an irrational real number.
Then the set $S_{f,g}(P,Q)$
is finite.
\end{thm}
\begin{proof}[{Proof of Theorem \ref{Theorem: DML type finiteness}}]
It is enough to prove that
once $f^{m_0}(P)=g^{n_0}(Q)$ is satisfied for some $m_0$ and $n_0$,
we have $f^{m+m_0}(P)\neq g^{n+n_0}(Q)$
for all $(m,n)\in \Z_{\geq 0}^2\setminus\{(0,0)\}$.
Therefore, it is enough to prove that for every point $R\in X(\var{\Q})$,
we have $f^m(R)\neq g^n(R)$ for all $(m,n) \in \Z_{\geq 0}^2\setminus \{(0,0)\}$.

Fix a Weil height function $h_H$ associated with $H$ so that $h_H\geq 1$.
By Theorem \ref{Theorem: order of height growth}, there are positive real numbers $C_0>0$ and $C_1>0$ satisfying
\begin{equation}\label{eqn: order}
C_0 n^{t_g(R)}\alpha_g(R)^n \leq h_H(g^n(R)) \leq C_1 n^{t_g(R)}\alpha_g(R)^n
\end{equation}
for all $n\in \Z_{\geq 1}$.
We set
\[
\widehat{\underline{h}}_{f,g,H}(R):= \liminf_{n\to \infty}\frac{\widehat{h}_{f,H}(g^n(R))}{n^{t_g(R)}\alpha_g(R)^n}.
\]
From \eqref{polarized order} and \eqref{eqn: order}, we have
\[
\widehat{\underline{h}}_{f,g,H}(R) = \liminf_{n\to \infty}\frac{h_H(g^n(R))}{n^{t_g(R)}\alpha_g(R)^n}\geq C_0 >0.
\]

Since the asymptotic behavior of $h_H(g^n(R))$ does not depend on
the choice of an ample divisor $H$ and a height function $h_H$,
one can see that
\begin{align}
h_H(g^n\circ f(R)) &= h_H(f\circ g^n(R))\\
&= h_{f^\ast H}(g^n(R))+O(1)\\
&\asymp n^{t_g(R)} \alpha_g(R)^n,
\end{align}
where the first equality follows from the commutativity of $f$ and $g$.
This asymptotic equality means that
we have $t_g(f(R))=t_g(R)$ and $\alpha_g(f(R))=\alpha_g(R)$.
Hence the functional equations
\begin{align}
\widehat{\underline{h}}_{f,g,H}(f(R)) &= \alpha_f(R)\widehat{\underline{h}}_{f,g,H}(R),\\
\widehat{\underline{h}}_{f,g,H}(g(R)) &= \alpha_g(R)\widehat{\underline{h}}_{f,g,H}(R)
\end{align}
hold (see Remark \ref{Remark: polarized case}).
Thus, if we have $f^m(R)=g^n(R)$,
we get
\begin{align}
\alpha_f(R)^m \widehat{\underline{h}}_{f,g,H}(R)
&= \widehat{\underline{h}}_{f,g,H}(f^m(R))\\
&= \widehat{\underline{h}}_{f,g,H}(g^n(R))\\
&= \alpha_g(R)^n \widehat{\underline{h}}_{f,g,H}(R).
\end{align}
Hence $\alpha_f(R)^m=\alpha_g(R)^n$.
This equality can hold only when $m=n=0$
since we are assuming $\log_{\alpha_f(R)}\alpha_g(R)$ is an irrational real number.
\end{proof}

\section*{Acknowledgments}
The author is grateful for the Top Global University project for Kyoto University
(abbrev.\ KTGU project).
With the support of the KTGU project, the author had a chance
to visit Professor Joseph H. Silverman during February-April, 2017
and October-November 2017.
The author would appreciate his hospitality when
the author was staying at Brown University.
These were great opportunities to discuss the research and
to consider the future works.
The author would like to thank Professor Tetsushi Ito
for carefully reading an early version of this paper and
pointing out some inaccuracies.
The author would also like to thank Takahiro Shibata who suggested
a variant of the Dynamical Mordell-Lang conjecture to the author.

\end{document}